\documentclass[11pt]{amsart}

\usepackage[utf8]{inputenc}
\usepackage{amsmath}
\usepackage{amsfonts}
\usepackage{amssymb}
\usepackage{amsthm}
\usepackage{csquotes}
\usepackage[english]{babel}
\usepackage[shortlabels]{enumitem}
\usepackage{lmodern}
\usepackage{cite}

\usepackage{color}

\newtheorem{theorem}{Theorem}[section]

\newtheorem{corollary}[theorem]{Corollary}
\newtheorem{lemma}[theorem]{Lemma}
\newtheorem{proposition}[theorem]{Proposition}

\theoremstyle{definition}

\theoremstyle{remark} \theoremstyle{remark}
\newtheorem{remark}[theorem]{Remark}
\newtheorem{example}[theorem]{Example}

\newtheorem*{theorem*}{Theorem}

\DeclareMathOperator{\im}{im}
\DeclareMathOperator{\id}{id}
\newcommand{\di}{d}

\newcommand{\summ}{\sum_{\alpha=1}^{s}}

\newcommand{\etaa}{\eta_{\alpha}}

\newcommand{\xib}{\xi_{\beta}}
\newcommand{\xia}{\xi_{\alpha}}

\newcommand{\ha}{h_{\alpha}}
\newcommand{\hb}{h_{\beta}}

\newcommand{\ka}{k_{\alpha}}
\newcommand{\cF}{{\mathcal F}}

\newcommand{\cL}{{\mathcal L}}
\newcommand{\mft}{{\mathfrak t}}

\newcommand{\RR}{{\mathbb R}}
\newcommand{\CC}{{\mathbb C}}
\newcommand{\SO}{{\mathrm{SO}}}

\numberwithin{equation}{section}

\begin{document}

\title[On the topology of metric $f$-$K$-contact manifolds]{On the topology of metric $f$-$K$-contact manifolds}

\author[O. Goertsches]{Oliver Goertsches}
 \address{Philipps Universit\"at Marburg, Fachbereich Mathematik und Informatik, Hans-Meerwein-Straße, 35032 Marburg, Germany}
 \email{goertsch@mathematik.uni-marburg.de}
 
\author[E. Loiudice]{Eugenia Loiudice}
 \email{loiudice@mathematik.uni-marburg.de}


\begin{abstract}
We observe that the class of metric $f$-$K$-contact manifolds, which naturally contains that of $K$-contact manifolds, is closed under forming mapping tori of automorphisms of the structure. We show that the de Rham cohomology of compact metric $f$-$K$-contact manifolds naturally splits off an exterior algebra, and relate the closed leaves of the characteristic foliation to its basic cohomology.
\end{abstract}



\maketitle

\section{Introduction}
An $f$-structure on a smooth manifold is a $(1,1)$-tensor $f$ of constant rank, satisfying $f^3+f=0$. This notion was introduced by Yano in \cite{yano} and generalizes both the notion of almost complex and of almost contact structure. The rank of $f$ is always even, and if maximal, then $f$ is either an almost complex or an almost contact structure (see \cite{yano} and Section~\ref{sec:preliminaries} for more details). $f$-structures with non-maximal rank (in particular with $\dim \ker (f)=2$) arise naturally when studying hypersurfaces of almost contact manifolds (see Blair--Ludden \cite{blair-ludden}). 

An analogue of Hermitian structures on almost complex manifolds and of contact metric structures on almost contact manifolds was introduced on the class of $f$-manifolds by Blair \cite{blair70}. A metric $f$-contact manifold is a $f$-manifold $(M^{2n+s},f)$ endowed with $s$ vector fields $\xi_1,\dots,\xi_s$, $s$ one forms $\eta_1,\dots,\eta_s$ and a Riemannian metric $g$ such that:
$$\etaa(\xib)=\delta_{\alpha}^{\beta}, \; f(\xia)=0, \; \etaa \circ f=0, \; f^2=-\id+\summ \etaa \otimes \xia, 
$$
$$d \etaa (X,Y)=g(X,fY),\; g(fX,fY)=g(X,Y)-\summ \etaa(X)\etaa(Y),$$
for every $\alpha, \beta \in \{1,\dots,s\}$ and  $X,Y \in TM$, where $\delta_\alpha^\beta$ is the Kronecker delta.

The Riemannian geometry of such manifolds was  studied intensively by various authors. We recall here some aspects of metric $f$-contact manifolds with $s\geq 2$ that are very different from the metric contact setting (i.e., when $s=1$).  
Blair \cite{blair70} showed that there are no $S$-manifolds (i.e., normal metric $f$-contact manifolds, see Section~\ref{sec:preliminaries}) $M^{2n+s}$ with $s\geq 2$ of constant strictly positive curvature.
Moreover Dileo--Lotta \cite{lotta-dileo} proved the non-existence of compact, simply connected, $S$-manifolds $M^{2n+s}$ with $s\geq 2$. 
Obviously the situation in the Sasakian setting (i.e., when $s=1$) is different.

\medskip
In Section~\ref{sec:cohomology} we prove a splitting theorem for the de Rham cohomology of metric $f$-$K$-contact manifolds, i.e., metric $f$-contact manifold $M$ whose characteristic vector fields $\xi_1,\dots,\xi_s$ are Killing. (For $s=1$ one obtains the well-known notion of a $K$-contact manifold.)
\begin{theorem}
For any compact metric $f$-$K$-contact manifold $M$ there is an isomorphism of $\Lambda(\RR^{s-1})$-algebras
\[
H^*(M)\cong \Lambda(\RR^{s-1})\otimes H^*(M,\cF_{s-1}).
\]
\end{theorem}

Here $\cF_{s-1}$ denotes the Riemannian foliation on $M$ determined by the Killing vector fields $\xi_1,\dots,\xi_{s-1}$, and $H^*(M,\cF_{s-1})$ the associated basic cohomology. The above mentioned result from \cite{lotta-dileo} is a direct consequence of this.

\medskip

In Section~\ref{sec:mappingtori} we describe a new method to construct examples of (compact) metric $f$-($K$-)contact manifolds.
Starting from any metric $f$-contact manifold $M$, we construct explicitly a metric $f$-contact structure on the mapping torus $M_\phi$ of any automorphism $\phi$ of the metric $f$-contact structure on $M$. This construction respects the subclasses of metric $f$-$K$-contact manifolds and of $S$-manifolds. We remark that this behavior is quite unusual; indeed, most geometric classes of manifolds are not preserved by forming mapping tori of automorphisms, see Remark \ref{rem:formingmappingtorus}.

\medskip
In Section~\ref{sec:Morse} and \ref{sec:closedleaves} we apply results from  \cite{GoertschesToeben} to relate the closed leaves of the characteristic foliation $\cF$ given by the characteristic vector fields $\xi_1,\dots,\xi_s$ on a metric $f$-$K$-contact manifold $M$ to the basic cohomology $H^*(M,\cF)$. We generalizes results of \cite{GoeNozToe} in the $K$-contact case. The main tool is the torus $T$ given by the closure of the flows of the characteristic vector fields in the isometry group of $M$, and a $T$-invariant Morse-Bott function $S$ whose critical set $C$ is equal to the union of closed leaves of $\cF$. This function generalizes a generic component of the contact momentum map in the $K$-contact setting, see \cite[Section 4]{rukimbira}. We obtain:
\begin{theorem}
 We have $\dim_{\RR} H^*(M,\cF) = \dim_{\RR} H^*(C,\cF)$. If $C$ consists of only finitely many  closed leaves of $\cF$, then $\dim_{\RR} H^*(M,\cF)$ is equal to the number of closed leaves of $\cF$.
\end{theorem}
We prove moreover the following
\begin{theorem}
The characteristic foliation of a compact metric $f$-$K$-contact manifold $M^{2n+s}$ has at least $n+1$ closed leaves.  If it has only finitely many closed leaves, then the following conditions are equivalent:
\begin{itemize}
\item The number of closed leaves of $\cF$ is $n+1$.
\item The basic cohomology $H^*(M,\cF)$ is that of $\CC P^n$, i.e., 
\[
H^*(M,\cF)= \RR[\omega]/([\omega^{n+1}]).
\]
\item The basic cohomology $H^*(M,\cF_{s-1})$ is that of a $2n+1$-dimensional sphere.
\item $M$ has the real cohomology ring of $S^{2n+1}\times T^{s-1}$.
\end{itemize}
\end{theorem}
As a consequence we obtain that any automorphism of the $K$-contact structure on a $K$-manifold $M^{2n+1}$ which has exactly $n+1$ closed orbits sends every closed Reeb orbit to itself (see Corollary~\ref{cor:K}).\\

\noindent {\it Acknowledgements.} We thank Antonio De Nicola for valuable comments on a previous version of the paper.

\section{metric $f$-manifolds}\label{sec:preliminaries}

A \emph{$f$-structure} on a smooth manifold $M^{2n+s}$ is a $(1,1)$ tensor $f$ of constant rank and such that $f^3+f=0$. Given such a structure, the tangent bundle of $M$ splits into two complementary subbundles $\im(f)$ and $\ker(f)$; moreover
$$
f^2|_{\im (f)}=-\id_{\im (f)},
$$
and thus the rank of $f$ is even, say $2n$ (cf. \cite{yano}).
If $\ker (f)$ is parallelizable then we fix $s$ global vector fields $\xi_1,\dots\xi_s$ on $M$ which span the kernel of $f$. Let $\eta_1,\dots,\eta_s$ be the $1$-forms determined by 
$$\etaa(\xib)=\delta_{\alpha}^{\beta}, \; \etaa \circ f=0.$$ 
Then we have:
\begin{equation*}
 f^2=-\id+\summ \etaa \otimes \xia.
\end{equation*}

In particular for $s=0$ or $s=1$ we have that $M^{2n+s}$ is an almost complex or respectively an almost contact manifold. If in addiction the structure tensors $f,\xia,\etaa$ satisfy the \emph{normality} condition: 
\begin{equation*}
 [f,f]+2\summ d\etaa \otimes \xia =0,
\end{equation*}
where $[f,f]$ denotes the Nijenhuis torsion of $f$, then $(M,f,\xia,\etaa)$ is called \emph{normal}, and for $s=0$ or $s=1$ we have that $M$ is a complex manifold or respectively a normal almost contact manifold.

It is well-known that a manifold $M^{2n+s}$ admitting an $f$-structure with parallelizable kernel always admits a \emph{compatible} metric, that is a Riemannian metric $g$ satisfying
\begin{equation*} 
 g(fX,fY)=g(X,Y)-\summ \etaa(X)\etaa(Y),
\end{equation*}
for every $X,Y \in TM$. The manifold $M^{2n+s}$ together with the structure tensors $(f,\xia,\etaa,g)$ as above is called a \emph{metric $f$-manifold}, and the $2$-form defined by:
$$
\omega(X,Y):=g(X,fY), \; X,Y \in TM
$$
is the \emph{fundamental $2$-form} of $M^{2n+s}$. 
A \emph{metric $f$-contact manifold} is a metric $f$-manifold $(M^{2n+s},f,\xia,\etaa,g)$ with $s>0$ such that 
\begin{equation*}
 d \etaa =\omega, 
\end{equation*}
for every  $\alpha \in \{1,\dots,s\}$. If a metric $f$-contact manifold is normal, then it is called a \emph{$S$-manifold}.

We observe that for $s=1$, the notion of metric $f$-contact manifold (resp.\ $S$-manifold) coincides with the notion of contact metric manifold (resp.\ Sasakian manifold).

\medskip 

We remark that one can construct metric $f$-contact structures on a manifold $M^{2n+s}$ starting from $s$ one-forms $\etaa$ on $M^{2n+s}$ satisfying some non-degeneracy condition \cite[Theorem 3.1]{diterlizzi}:
\begin{theorem}
Let $M$ be a smooth manifold of dimension $2n+s$ admitting $s$ one-forms $\eta_1,\dots,\eta_s$ such that 
 $d \eta_1= \dots =d \eta_s$
 is a $2$-form of constant rank $2n$ and 
$$\eta_1 \wedge \dots \wedge \eta_s \wedge (d \eta_1)^n\neq 0.$$ Then there exists a metric $f$-contact structure $(f,\xi_1,\dots,\xi_s,\eta_1,\dots,\eta_s,g)$ on $M$, where $\xi_1,\dots,\xi_s$ are the unique vector fields on $M$ such that $\etaa(\xib)=\delta_{\alpha}^{\beta}$ and $i_{\xib}d\etaa=0$ for every $\alpha,\beta \in \{1,\dots,s\}$.
 \end{theorem}
This result generalizes the well known construction of contact metric structures on a odd-dimensional manifold endowed with a contact form, see for instance \cite[Theorem 4.4]{blairbook}.

 \medskip
In the following we recall some useful properties of metric $f$-contact manifolds obtained in \cite{cabrerizo-fern} and \cite{duggal-ianus-past}.
Let $(M, f, \xia, \etaa, g)$ be a metric $f$-contact manifold. Then the operators $$\ha:=\frac{1}{2}\mathcal{L}_{\xia}f, \; \alpha \in\{1,\dots,s\}$$
where $\mathcal{L}_{\xia}$ denotes the Lie derivative relative to $\xia$, are self-adjoint and anticommute with $f$. Moreover, for every $\alpha,\beta\in\{1,\dots s\}$ and $X\in TM$ we have
\begin{equation}\label{hxi}
 \ha \xib=0
\end{equation}
by \cite[Proposition 2.3]{cabrerizo-fern},
\begin{equation}\label{nabla f}
 \nabla_{\xia}f=0
\end{equation}
by \cite[Equation (2.4)]{cabrerizo-fern},
\begin{equation}\label{ab}
 [\xia,\xib]=0
\end{equation}
by \cite[Corollary 2.4]{cabrerizo-fern} and
\begin{equation}\label{nabla xi}
 \nabla_X \xib =-fX-f\hb X
\end{equation}
by \cite[Proposition~2.4]{duggal-ianus-past}.

\medskip

A metric $f$-contact manifold whose characteristic vector fields $\xi_1,\dots,\xi_s$ are Killing is called a \emph{metric $f$-$K$-contact} manifold. The following theorem is proved in \cite[Theorem 2.6]{cabrerizo-fern}.
\begin{theorem}\label{th killing}
 Let $(M,f,\xi_1,\dots,\xi_s,\eta_1,\dots,\eta_s,g)$ be a metric $f$-contact manifold. Then, for any $\alpha \in \{1,\dots,s\}$, the vector field $\xia$ is Killing if and only if $\ha=0$.
\end{theorem}

Hence, if $(M,f,\xia,\etaa,g)$ is metric $f$-$K$-contact manifold, Equation \eqref{nabla xi} becomes 
\begin{equation}\label{Xxi}
 \nabla_X \xia =-fX.
\end{equation}
Using \eqref{Xxi} we conclude that the curvature tensor field $R$ of $M$ satisfies 
\begin{equation}\label{l}
 R(X,\xia)Y=\nabla_X\nabla_Y{\xia}-\nabla_{\nabla_XY}{\xia}=-(\nabla_Xf)Y,
\end{equation}
for each $X,Y\in TM$ and $\alpha \in \{1,\dots,s\}$, where we used\cite[Proposition~8.1.3]{petersen} for the first equality.

\section{Mapping tori of metric $f$-$K$-contact manifolds} \label{sec:mappingtori}

In \cite[Proposition 4.1]{lotta-dileo} it was shown that the product of a Sasakian manifold with an Abelian Lie group always admits the structure of an $S$-manifold. In this section we use the same idea to show that the classes of metric $f$-contact, metric $f$-$K$-contact, and $S$-manifolds are closed under forming the mapping torus with respect to automorphisms of the structure. We begin by describing an explicit induced structure on the product with the real line.

Let $(M^{2n+s},f,\xi_1,\dots,\xi_s, \eta_1,\dots,\eta_s, g)$ be a metric $f$-contact manifold with fundamental form $\omega$. On the product manifold $M\times\mathbb{R}$ we define a $(1,1)$ tensor $\bar{f}$ and $s+1$ one-forms $\bar{\eta}_{1},\dots ,\bar{\eta}_{s+1}$ by
 \begin{equation*}
 \begin{aligned}
  &\bar{f}(X)=f(X), \; \bar{f}\left(\frac{d}{dt}\right)=0, \\
  &\bar{\eta}_{\alpha}(X)=\eta_{\alpha}(X), \; \bar{\eta}_{\alpha}\left(\frac{d}{dt}\right)=0, \; \alpha=1,\dots s, \\
  &\bar{\eta}_{s+1}(X)=\frac{1}{s}(\eta_1(X)+\dots +\eta_s(X)), \; \bar{\eta}_{s+1}\left(\frac{d}{dt}\right)=1,\\
 \end{aligned}
 \end{equation*}
for each $X\in TM$ and where $\frac{d}{dt}$ denotes the standard coordinate vector field on $\mathbb{R}$.
 We have that $\bar{f}$ is an $f$-structure on $M\times \mathbb{R}$, $\im(\bar{f})=\bigcap_{\alpha}\ker \bar{\eta}_{\alpha}= \im(f)$, $d \bar{\eta}_{1}=\dots= d \bar{\eta}_{s+1}=\pi_1^*\omega=:\bar{\omega}$, where $\pi_1:M\times \mathbb{R}\to M$ is the projection on the first component. We  have $\bar{\omega}\wedge \bar{\eta}_1\dots \wedge \bar{\eta}_{s+1}\neq 0$.
The vector fields
\begin{equation*}
 \begin{aligned}
  &\bar{\xi}_{\alpha} :=\xia-\frac{1}{s}\frac{d}{dt}, \quad \alpha = 1,\dots,s, \\
  &\bar{\xi}_{s+1} :=\frac{d}{dt},
 \end{aligned}
\end{equation*}
are dual to $\bar{\eta}_{1}, \dots ,\bar{\eta}_{s+1}$ and generate the kernel of $\bar{f}$. We consider moreover the Riemannian metric $\bar{g}$ defined by
\begin{equation*}
\begin{aligned}
 & \bar{g}(X,Y)= g(X,Y), \; \bar{g}(X,\bar{\xi}_{\alpha})=0,\; \bar{g}(\bar{\xi}_{\alpha}, \bar{\xi}_{\beta})=\delta_{\alpha}^{\beta},
 \end{aligned}
\end{equation*}
for each  $X,Y \in \im(f)$ and ${\alpha}, {\beta} \in \{1,\dots,s+1\}$. It is easy to check that $(\bar{f}, \bar{\xi}_{1}, \dots,\bar{\xi}_{s+1}, \bar{\eta}_{1},\dots, \bar{\eta}_{s+1}, \bar{g})$ is a metric $f$-contact structure on $M\times \mathbb{R}$. 

\medskip

Let $V$ be a local vector field tangent to $\im (f)$. Observe that, for each $\alpha\in \{1,\dots,s\}$ 
\begin{equation*}
\begin{aligned}
 2\bar{h}_{\alpha}(V)&= \left[\xia - \frac{1}{s}\frac{d}{dt},\bar{f}V\right]-\bar{f}\left[\xia - \frac{1}{s}\frac{d}{dt},V\right]\\
                    &=[\xia ,{f}V]-f[\xia , V]\\
                    &=2\ha(V)\\
\end{aligned}
\end{equation*}
and 
\begin{equation*}
 \bar{h}_{s+1}(V)=0.
\end{equation*}
Then, using Theorem~\ref{th killing} and Equation \eqref{hxi},
we obtain that $(f,\xia,\etaa,g)$ is a metric $f$-$K$-contact structure on $M$ if and only if $(\bar{f}, \bar{\xi}_{\alpha},\bar{\eta}_{\alpha},\bar{g})$ is a metric $f$-$K$-contact structure on $M\times \mathbb{R}$.

\medskip

Now consider two local vector fields $V,W$ tangent to $\im (f)$ and $\beta \in \{1,\dots,s\}$. We have:
\begin{equation*}
\begin{aligned}
 ([\bar{f},\bar{f}]+2 \sum_{\alpha=1}^{s+1} d\bar{\eta}_{\alpha} \otimes \bar{\xi}_{\alpha})(V,W) &= [f,f](V,W)+2 \omega(V,W)({\xi}_1+\dots+{\xi}_{s})\\
                                                                                                                                                                                                    &= ([f,f]+2 \summ d \etaa \otimes \xia )(V,W),\\
([\bar{f},\bar{f}]+2 \sum_{\alpha=1}^{s+1} d\bar{\eta}_{\alpha} \otimes \bar{\xi}_{\alpha})(V,\bar{\xi}_{\beta}) &= \bar{f}^2[V,\bar{\xi}_{\beta}]-\bar{f}[\bar{f}V,\bar{\xi}_{\beta}]\\
                                                                                                                 &= \bar{f}^2[V,{\xi}_{\beta}]-\bar{f}[{f}V,{\xi}_{\beta}]\\
                                                                                                                 &= ([f,f]+2 \summ d \etaa \otimes \xia )(V,\xib),\\
([\bar{f},\bar{f}]+2 \sum_{\alpha=1}^{s+1} d\bar{\eta}_{\alpha} \otimes \bar{\xi}_{\alpha})(V,\bar{\xi}_{s+1})   &= \bar{f}^2[V,\bar{\xi}_{s+1}]-\bar{f}[\bar{f}V,\bar{\xi}_{s+1}] =0.                                                                                                        
\end{aligned}
\end{equation*}
Then, since $([\bar{f},\bar{f}]+2 \omega \otimes \sum_{\alpha=1}^{s+1}\bar{\xi}_{\alpha})(\bar{\xi}_{\beta},\bar{\xi}_{\gamma})= 0$, for each $\beta, \gamma \in \{1,\dots,s+1\}$, we have that $(M, f,\xia,\etaa,g)$ is a $S$-manifold if and only if $(M\times\mathbb{R}, \bar{f},\bar{\xi}_{\alpha},\bar{\eta}_{\alpha},\bar{g})$ is a $S$-manifold.

\medskip

Summarizing, if $(f,\xia,\etaa,g)$ is a metric $f$-$K$-contact structure (resp. $S$-structure) on $M$, then the induced structure tensors $(\bar{f}, \bar{\xi}_{\alpha},\bar{\eta}_{\alpha},\bar{g})$ on the product manifold $M\times \mathbb{R}$ determine a metric $f$-$K$-contact structure (resp. $S$-structure) on $M\times \mathbb{R}$.

\begin{remark} 
Another natural choice to construct a metric $f$-structure on the product manifold $M\times \mathbb{R}$, is to consider on $M\times \mathbb{R}$ the product metric 
$$\bar{g}:=\pi_1^*(g)+\pi_2^*(dt^2),$$ 
where $\pi_1$ and $\pi_2$ are the projections from $M\times \mathbb{R}$ on $M$ and $\mathbb{R}$ respectively,  and the tensors $\bar{f},\bar{\xi}_{\alpha}, \bar{\eta}_{\alpha}$ defined by:
\begin{equation*}
\begin{aligned}
 & \bar{f}(X)=f(X), \; \bar{f}\left(\frac{d}{dt}\right)=0, \; \bar{\xi}_{\alpha}=\xia, \; \bar{\xi}_{s+1}=\frac{d}{dt}, \\
 & \bar{\eta}_{\alpha}(X)=\etaa(X), \;  \bar{\eta}_{\alpha}\left(\frac{d}{dt}\right)=0, \; \bar{\eta}_{s+1}(X)=0, \;  \bar{\eta}_{s+1}\left(\frac{d}{dt}\right)=1,
\end{aligned}
\end{equation*}

for every $X\in TM$ and $\alpha \in \{1,\dots,s\}$. One can easily check that, if $(f,\xia,\etaa,g)$ is a metric $f$-contact structure on $M$, then $(\bar{f}, \bar{\xi}_{\alpha},\bar{\eta}_{\alpha},\bar{g})$ is a metric $f$-structure on $M\times \mathbb{R}$; however, since $d\eta_{s+1}=0$, it is not a metric $f$-contact structure.

The construction above, generalized to warped products, was used in \cite[Example 3.3]{carriazo-fernandez} to produce examples of generalized $S$-space forms from generalized Sasakian space-forms (see \cite{alegre-blair-carriazo}).  
\end{remark}

\medskip
We now show that the structure constructed above descends to mapping tori of automorphisms. Recall that, for a diffeomorphism $\phi:M\rightarrow M$ of a manifold $M$ the \emph{mapping torus} $M_{\phi}$ of $(M,\phi)$ is the quotient space $(M\times \mathbb{R})/\mathbb{Z}$, where the free and properly discontinuous $\mathbb{Z}$-action on the product space $M\times \mathbb{R}$ is given by
\begin{equation*}
 m \cdot (p,t)=(\phi^m(p),t+m).
\end{equation*}
Let now $(M,f,\xia,\etaa,g)$ be a metric $f$-contact manifold and $\phi:M\rightarrow M$ an automorphism of the metric $f$-structure.
 Observe that the diffeomorphism
 $$\rho_m: M\times \mathbb{R}\rightarrow M\times \mathbb{R}; \; (p,t)\mapsto (\phi^m(p),t+m),$$ 
 $ m\in \mathbb{Z}$, preserves the structure tensors $ \bar{f}, \bar{\xi}_{\alpha},\bar{\eta}_{\alpha},\bar{g}$ on $M\times \mathbb{R}$ defined above. It follows that the tensors $\bar{f}, \bar{\xi}_{\alpha},\bar{\eta}_{\alpha}, \bar{g}$ on $M\times \mathbb{R}$ descend to $M_{\phi}$, making it a metric $f$-contact manifold.    
We have moreover that, if $(M,f,\xia,\etaa,g)$ is a (compact) metric $f$-$K$-contact manifold (or a $S$-manifold), then $M_{\phi}$ with the induced structure is a (compact) metric $f$-$K$-contact manifold (or respectively a $S$-manifold).

\begin{remark}\label{rem:formingmappingtorus}
Most geometric classes of manifolds are not preserved by forming mapping tori of automorphisms. For instance, the mapping torus of a symplectomorphism of a symplectic manifold naturally is a cosymplectic manifold, and that of a holomorphic isometry of a K\"ahler manifold is a co-K\"ahler manifold (see \cite[Lemmata 1 and 4]{Li}). The mapping torus of a strict contactomorphism of a contact manifold is a locally conformally symplectic manifold \cite[Example 2.4]{BazzMarrero}, and an automorphism of a Sasakian manifold induces a Vaisman structure on the mapping torus. From this point of view, metric $f$-$K$-contact structures behave in a rather unusual way.
\end{remark}

\section{Cohomology of metric $f$-$K$-contact manifolds}\label{sec:cohomology}
Throughout this section, we consider a compact metric $f$-$K$-contact manifold $(M^{2n+s},f,\xia, \etaa, g)$ with $s\geq 2$. We recall that $\omega=d\etaa$ for all $\alpha=1,\ldots, s$. Our first observation is
\begin{lemma}
For all $\alpha\neq \beta$, the one-form $\eta_\alpha-\eta_\beta$ defines a nonzero element in $H^1(M)$. 
\end{lemma}
\begin{proof}
The one-form $\eta_\alpha-\eta_\beta$ is closed because $d\eta_\alpha=\omega$ for all $\alpha$. If it was exact, then $\eta_\alpha-\eta_\beta=dh$ for a real-valued function $h$ on $M$. As $M$ is compact, $h$ has a critical point, so that $\eta_\alpha-\eta_\beta$ has a zero. But $(\eta_\alpha-\eta_\beta)(\xia) = 1$ on all of $M$. 
\end{proof}

Let $M$ be a compact manifold. For a foliation $\cF$ on $M$, we will consider its basic cohomology $H^*(M,\cF)$, which is by definition the cohomology of the subcomplex
\[
\Omega(M,\cF)=\{\sigma\in \Omega(M)\mid i_{X}\sigma = \cL_{X}\sigma = 0 \textrm{ for all } X\in \Xi(\cF)\}
\]
of the de Rham complex $(\Omega(M),d)$, where $\Xi(\cF)$ denotes the space of vector fields tangent to $\cF$.

As the $\xi_\alpha$ are commuting Killing vector fields, they define an $s$-dimensional Riemannian foliation $\cF$ on $M$, which we call the \emph{characteristic foliation} of $M$.  We will also make use of the Riemannian foliations on $M$ spanned by the Killing vector fields $\xi_1,\ldots,\xi_k$, for $k=1,\ldots,s$, which we denote by $\cF_k$. The leaf dimension of $\cF_k$ is $k$; we have $\cF_s=\cF$, and we denote by $\cF_0$ the trivial foliation by points. 

Obviously, the leaves of $\cF_k$ are contained in those of $\cF_{k+1}$, for all $k=0,\ldots,s-1$.

\begin{proposition}\label{prop:gysinsequences}
We have short exact sequences
\[
0 \longrightarrow H^*(M,\cF_{k+1}) \longrightarrow H^*(M,\cF_{k}) \longrightarrow H^{*-1}(M,\cF_{k+1}) \longrightarrow 0,
\]
for all $k=0,\ldots,s-2$, as well as
\begin{align*}
\cdots \longrightarrow H^p(M,\cF)&\longrightarrow H^p(M,\cF_{s-1})\\
&\quad \longrightarrow H^{p-1}(M,\cF)\overset{\delta}\longrightarrow H^{p+1}(M,\cF)\longrightarrow \cdots, 
\end{align*}
where the connecting homomorphism $\delta$ is given by $\delta([\sigma])=[\omega\wedge \sigma]$.
\end{proposition}
\begin{proof}
This follows from a variant of the Gysin sequence for pairs of foliations, whose proof is analogous to \cite[Proposition 7.2.1]{BoyerGalicki}: Consider, for any $k=0,\ldots,s-1$, the short exact sequence of complexes
\[
0\longrightarrow \Omega^*(M,\cF_{k+1})\longrightarrow \Omega^*(M,\cF_{k})^{T_{k+1}}\overset{i_{\xi_{k+1}}}\longrightarrow \Omega^{*-1}(M,\cF_{k+1})\longrightarrow 0,
\]
where $T_{k+1}$ is the torus defined as the closure of the flows of the Killing vector fields defining $\cF_{k+1}$, i.e., $\xi_1,\ldots,\xi_{k+1}$, in the isometry group of $M$. The first map in the sequence is the natural inclusion. One observes that the inclusion $\Omega^*(M,\cF_k)^{T_{k+1}}\subset \Omega^*(M,\cF_k)$ induces an isomorphism in cohomology. (It is shown in \cite[\S 9.1, Theorem 1]{onishchik} that the averaging operator $\Omega^*(M)\to \Omega^*(M)^{T_{k+1}}$ to $\Omega^*(M,\cF_k)$ induces an isomorphism in cohomology, and one can restrict this operator to $\Omega^*(M,\cF_k)$. In a slightly different context, this argument was used also in \cite[Lemma 5.3]{BazzGoe}).

To understand the connecting homomorphism in the induced long exact sequence in cohomology one notes that for a given closed $\sigma\in \Omega^{p-1}(M,\cF_{k+1})$ a preimage under $i_{\xi_{k+1}}:\Omega^*(M,\cF_k)^{T_{k+1}}\to \Omega^{*-1}(M,\cF_{k+1})$ is given by $\eta_{k+1}\wedge \sigma$. This implies that $\delta([\sigma]) = [d(\eta_{k+1}\wedge \sigma)] = [\omega\wedge \sigma]$. If $k<s-1$, then $\omega=d\eta_s$ is exact in $\Omega^*(M,\cF_{k+1})$, so that the connecting homomorphism vanishes. For $k=s-1$ the form $\omega$ is not exact.
\end{proof}

We denote by $\Lambda(\RR^{s-1})$ the exterior algebra on $s-1$ generators, with generators in degree one. There is a natural homomorphism $\Lambda(\RR^{s-1})\to H^*(M)$ sending the standard basis vector $e_i$ to $[\eta_i-\eta_s]$, introducing on $H^*(M)$ the structure of a $\Lambda(\RR^{s-1})$-algebra.

\begin{theorem}\label{thm:cohomproduct}
 There is an isomorphism of $\Lambda(\RR^{s-1})$-algebras
\[
H^*(M)\cong \Lambda(\RR^{s-1})\otimes H^*(M,\cF_{s-1}).
\]
\end{theorem}
\begin{proof}
 The exact sequences in Proposition \ref{prop:gysinsequences} imply that the natural map $H^*(M,\cF_{s-1})\to H^*(M)$ is injective. We claim that $H^*(M,\cF_{s-1})$ generates $H^*(M)$ freely as a $\Lambda(\RR^{s-1})$-algebra. 

To see that the $\Lambda(\RR^{s-1})$-algebra morphism $\Lambda(\RR^{s-1})\otimes H^*(M,\cF_{s-1})\to H^*(M)$ is injective it suffices to show that for nonzero $[\sigma]\in H^*(M,\cF_{s-1})$ the element
\[
[(\eta_1-\eta_s)\wedge \cdots \wedge (\eta_{s-1}-\eta_s)\wedge\sigma]\in H^*(M)
\]
is nonzero. We can assume that the representative $\sigma$ is invariant under the torus $T$ generated by the flow of $\xi_1,\ldots,\xi_s$. Applying the composition $i_{\xi_{s-1}}\circ \cdots \circ i_{\xi_1}:H^*(M)\to H^*(M,\cF_1)\to \cdots \to H^*(M,\cF_{s-1})$ to this class, we get back the original nonzero element $[\sigma]\in H^*(M,\cF_{s-1})$. Hence the homomorphism is injective.

Surjectivity follows for dimensional reasons: the short exact sequences in Proposition \ref{prop:gysinsequences} imply that $\dim_{\RR} H^*(M) = 2^{s-1}\dim_{\RR} H^*(M,\cF_{s-1})$.
\end{proof}

\begin{remark}
In \cite[Corollary~4.3]{lotta-dileo} Dileo--Lotta showed the non-existence of simply connected, compact $S$-manifolds with $s\geqslant 2$. (Note that obviously for $s=1$ the result does not hold, as any odd dimensional sphere admits a Sasakian structure.) Theorem \ref{thm:cohomproduct} implies the same statement for metric $f$-$K$-contact manifolds:
\end{remark}
 \begin{corollary}
There are no compact, simply connected, metric $f$-$K$-contact manifolds $(M^{2n+s},f,\xia,\etaa,g)$, with $s\geq 2$.
             \end{corollary}

\begin{remark}
One can derive a cohomological splitting, similar to Theorem \ref{thm:cohomproduct}, from a Theorem of Chevalley, see \cite[\S IX.2, Theorem I]{Greub}. Concretely, given a compact metric $f$-$K$-contact manifold $(M,f,\xi_1,\dots \xi_s,\eta_1,\dots, \eta_s,g)$, we consider the Abelian Lie algebra 
\[
\mathfrak{g}= \{ \sum_{\alpha=1}^s a_\alpha \xi_\alpha \;|\; \sum_{\alpha=1}^s a_\alpha=0, \, a_\alpha\in \mathbb{R}\}
\]
as well as the foliation $\bar{\cF}$ it defines.

As the $\xi_{\alpha}$ are commuting Killing vector fields and $g(\xi_\alpha,\xi_\beta)=\delta_{\alpha}^{\beta}$ is constant for all $\alpha$ and $\beta$, we can apply \cite[Corollary~2.20]{deNicola} and obtain an algebraic connection $\chi: \mathfrak{g}^*\rightarrow \Omega^1(M)$ for the action of $\mathfrak{g}$: for an orthonormal basis $\{e_i\}$ of $\mathfrak{g}$ with dual basis $\{e^i\}$, we have
\[
\chi(e^i) = g(e_i,\cdot).
\]
Then, \cite[Theorem~2.21]{deNicola} yields a quasi-isomorphism of CDGAs
\[\Lambda (\mathfrak{g}^*)\otimes \Omega^*(M,{\bar{\cF}}) \longrightarrow \Omega_{\mathrm{bas}\, {\mathfrak{g}}}^*(M),
\]
where $\Omega_{\mathrm{bas}\, {\mathfrak{g}}}^*(M):=\{\omega\in \Omega(M) \;|\; \mathcal{L}_X\omega=0 \text{ for all } X \in \mathfrak{g}\}$ is the subcomplex of $\mathfrak{g}$-basic forms on $M$. Here, we consider on $\Omega_{\mathrm{bas}\, {\mathfrak{g}}}^*(M)$ the standard differential; see \cite[Sections 2.2 and 2.5]{deNicola} for the definition of the differential $d_{\bar{\chi}}$ on $\Lambda (\mathfrak{g}^*)\otimes \Omega^*(M,{\bar{\cF}})$. In our setting, as the forms $g(e_i,\cdot)$ are linear combinations of the closed one-forms $\eta_\alpha$, they are closed; hence, $\bar{\chi}=d\circ \chi=0$ and thus the differential $d_{\bar{\chi}}$ is just $1\otimes d$. This implies, together with the fact that the inclusion $\Omega_{\mathrm{bas}\, {\mathfrak{g}}}^*(M)\rightarrow \Omega^*(M)$ induces an isomorphism in cohomology (see \cite[\S 9.1, Theorem 1]{onishchik}), that we obtain an isomorphism
\[
\Lambda({\mathfrak{g}}^*)\otimes H^*(M,\bar{\cF})\longrightarrow H^*(M).
\]
\end{remark}
             
%
%
%
%
%

\section{Morse theory on metric $f$-$K$-contact manifolds}\label{sec:Morse}

In this section we construct, on any compact metric $f$-$K$-contact manifold, a Morse-Bott function whose critical set is the union of the closed leaves of the characteristic foliation. The construction and proof goes along the same lines as in the $K$-contact case, see \cite[Section 4]{rukimbira}. 

Let $(M^{2n+s},f,\xia, \etaa, g)$ be compact metric $f$-$K$-contact manifold. Observe that, as $\xi_1, \dots,\xi_s$ commute with each other, the closure in the isometry group $\text{Isom}(M,g)$ of $M$ of the subgroup generated by the flow of the characteristic vector fields  
\begin{equation*}
 T:=\overline{\langle\exp(t_1\xi_1),\dots, \exp(t_s\xi_s)\rangle}
\end{equation*}
is a connected, abelian Lie subgroup of $\text{Isom}(M,g)$, which is also compact being $M$ compact by hypothesis; hence $T$ is a torus. Let $Z\in Lie(T)=:\mft$ be a generic element, in the following sense: in every point $p\in M$ the isotropy Lie algebra $\mft_p$ is of dimension at most $\dim T - s$, as the elements $\xi_\alpha$ are never contained in it. We define
\[
\tilde \mft_p:= \mft_p \oplus \bigoplus_{\alpha} \RR  \xi_\alpha.
\]
As $M$ is compact, there are in total only finitely many distinct subspaces $\tilde \mft_p\subset \mft$. We choose $Z$ to satisfy
\[
Z\in \mft\setminus \bigcup_{p:\, \tilde\mft_p\neq \mft} \tilde\mft_p;
\]
note that this condition is void in case $\dim T=s$.

The fact that $[\xia,\xib]=0$ (see Equation \ref{ab}) implies the invariance of $\etaa$ under the flow of $\xib$ for each $\alpha$, $\beta \in \{1,\dots,s\}$; then, by continuity, any $\etaa$ is preserved by the $T$-action on $M$. In particular
\begin{equation}\label{LZ}
 \mathcal{L}_Z\etaa=0,
\end{equation}
for each $\alpha \in \{1,\dots,s\}$.

\medskip
Consider the real-valued map $S:M\rightarrow \mathbb{R}$, $p\mapsto \etaa(Z)(p)$. Using \eqref{LZ}, we have:
\begin{equation*}
 (\di S)(p)=\di(i_Z{\etaa})(p)=(i_Z\di{\etaa})(p)= \di \etaa(Z_p,\cdot).
\end{equation*}
Thus the critical set $C$ of $S$ consists of the points $p\in M$ such that $Z_p\in \bigoplus_{\alpha}\mathbb{R}(\xia)_p$. Observe that by our choice of $Z$ we have 
\[
C=\{p\in M \;|\; \dim T\cdot p=s\},
\]
which is the same as the union of the closed leaves of the characteristic foliation $\cF$ of $M$.

\begin{lemma}
 Let $N$ be a connected component of $C$ and $p\in N$. Consider the Killing vector field 
 \begin{equation}\label{delta}
  \delta=Z-\summ \ka \xia,
 \end{equation}
where $k_{\alpha}=\etaa(Z)(p)$, $\alpha\in \{1,\dots, s\}$, which vanishes along $N$. Then for all $v,w\in T_pM$ perpendicular to $N$ we have:
 \begin{enumerate}[(i)]
  \item $\nabla_v Z=-kf(v)+\nabla_v\delta$, where $\nabla_v\delta$ is a nonzero tangent vector perpendicular to $N$ and $k=\summ k_{\alpha}$. \label{i}
  \item $Hess_S(p)(v,w)=2g(R(\xia,v)Z_p,w)+2g(f(\nabla_vZ),w)$.  \label{ii}
  \item $Hess_S(p)(v,f(\nabla_v\delta))=2g(\nabla_v\delta, \nabla_v\delta)$.
  Therefore the Hessian of $S$ along $N$ is nondegenerate in directions perpendicular to $N$.
 \end{enumerate}
\end{lemma}

\begin{proof}
The $T$-isotropy Lie algebra is constant along the closed submanifold $N$; in fact, $N$ is equal to a connected component of the fixed point set of $T_p^0$, the identity component of the isotropy group $T_p$. It follows that $\mft = \mft_p \oplus \bigoplus_{\alpha=1}^s {\mathbb{R}} \xi_\alpha$, and the equality $Z=\delta + \sum_{\alpha=1}^s k_\alpha\xi_\alpha$  is precisely the decomposition of $Z$ according to this decomposition of $\mft$. This implies that $\nabla_v Z = -kf(v) + \nabla_v\delta$ (using \eqref{Xxi}) and that $\delta$ vanishes along all of $N$.

If $\nabla_v\delta=0$ then, because also $\delta(p)=0$ and $\delta$ is a Jacobi vector field along the geodesic $\gamma$ with initial velocity $v$, we have that $\delta$ vanishes along $\gamma$. On the other hand, by our choice of $Z$, in a neighborhood of $N$ the vector field $Z$ vanishes only in points of $N$. This implies that $\im(\gamma)\subset N$, contradicting the the fact that $v$ is perpendicular to $N$. To complete the proof of \ref{i}, consider a vector $X$ at $p$ tangent to $N$. As $N$ is a closed totally geodesic submanifold of $M$ (see \cite{kobayashi}) and $\delta$ is a Killing vector field which, restricted to $N$, is tangent to $N$, we have $g(\nabla_v\delta,X)=-g(\nabla_X\delta,v)=0$ for all $X$ tangent to $N$, so that $\nabla_v\delta$ is perpendicular to $N$.

To prove \ref{ii} consider $V$, $W$ local vector fields extending $v$, $w$ by parallel translation along geodesics emanating from $p$. Using  \eqref{Xxi}, \eqref{l} and the fact that $Z$ and $\xia$ are Killing vector fields commuting with each other we obtain:
\begin{equation*}
\begin{aligned}
 Hess_S(p)(v,w) &=VW(S)(p)\\
                &= VW(g(\xia, Z))(p)\\
                &= V( g(\nabla_W\xia,Z)+g(\xia, \nabla_WZ))(p)\\
                &= V(g(-fW, Z)-g(\nabla_{\xia}Z,W))(p)\\
                &= 2V(g(W,fZ))(p)\\
                &= 2(g(\nabla_VW, fZ)+g((\nabla_Vf)(Z),W)+g(W,f(\nabla_VZ)))(p)\\
                &= 2g(R(\xia,v)Z_p,w)+2g(f(\nabla_vZ),w).
\end{aligned}
\end{equation*}
We also compute 
\begin{equation*}
\begin{aligned}
R(\xi_\alpha,v)\xi_\beta &= -R(v,\xi_\alpha)\xi_\beta = (\nabla_v f)\xi_\alpha = -f(\nabla_v \xi_\alpha)= f^2(v) = -v,
\end{aligned}
\end{equation*}
where we used \eqref{l} in the second and \eqref{Xxi} in the fourth equality. The last equality is true because $N$ is $T$-invariant and hence $v$ is contained in the image of $f$. Now, using this information, we continue the computation above, taking $w=f(\nabla_v\delta)$: 
\begin{equation*}
\begin{aligned}
 Hess_S(p)(v,f(\nabla_v\delta)) &= 2g(R(\xia,v)Z_p,f(\nabla_v\delta))+2g(f(\nabla_vZ),f(\nabla_v\delta))\\
                               &= 2g(\sum_{\beta=1}^sk_{\beta}R(\xia,v)\xib,f(\nabla_v\delta))+2g(\nabla_vZ,\nabla_v\delta)\\
                               &= 2kg(fv, \nabla_v\delta)+2g(-kf(v)+\nabla_v\delta,\nabla_v\delta)\\
                               &= 2g(\nabla_v\delta, \nabla_v\delta)\neq 0.
\end{aligned}
\end{equation*}
\end{proof}
In this computation we used that $\delta$ vanishes in $p$ for the second equality,  that $\nabla_v\delta$ is perpendicular to $N$ in the second and third equality, and the identity $\nabla_v Z=-kf(v)+\nabla_v\delta$ from (i) in the third equality.

Therefore it follows:
\begin{proposition}\label{prop:invariantmorsebott}
 The function $S$ is a $T$-invariant Morse-Bott function with critical set $C$.
\end{proposition}

\section{Closed leaves of the characteristic foliation} \label{sec:closedleaves}
In this section  we relate the ordinary and basic cohomology of a compact metric $f$-$K$-contact manifold $(M^{2n+s},f,\xia,\etaa,g)$ to the union $C$ of the closed leaves of the characteristic foliation $\cF$. This generalizes results from \cite{rukimbira} and \cite{GoeNozToe}. 

As usual we denote the fundamental $2$-form of $M^{2n+s}$ by $\omega$. The function $S$ considered in Section \ref{sec:Morse} is $\cF$-basic, i.e., constant along leaves of $\cF$. Because of Proposition \ref{prop:invariantmorsebott}, \cite[Theorems 6.3 and 6.4]{GoertschesToeben} are applicable and we obtain:
\begin{theorem}
 We have the following equality of Poincar\'e polynomials:
\[
P_t(M,\cF) = \sum_N t^{\lambda_N} P_t(N,\cF),
\]
where $N$ runs over the connected components of $C$, and $\lambda_N$ is the index of $N$, i.e., the rank of the negative normal bundle of $N$ with respect to $S$.
\end{theorem}
Here, $P_t(M,\cF)=\sum t^k \dim H^k(M,\cF)$ is the Poincar\'e series of $H^*(M,\cF)$, and analogously for $(N,\cF)$. In particular we obtain by evaluating this equation at $t=1$:
\begin{corollary}\label{cor:basiccohomclosedleaves}
We have $\dim_{\RR} H^*(M,\cF) = \dim_{\RR} H^*(C,\cF)$. If $C$ consists of only finitely many  closed leaves of $\cF$, then $\dim_{\RR} H^*(M,\cF)$ is equal to the number of closed leaves of $\cF$.
\end{corollary}

Recall that an $s$-form $\eta$ on a foliation $\cF$ of leaf dimension $s$ is called \emph{relatively closed} if $d\eta(v_1,\ldots,v_{s+1})=0$ whenever $s$ of the $s+1$ tangent vectors $v_i$ are tangent to $\cF$. It is well-known that for a relatively closed $s$-form $\eta$ on an oriented manifold $M$ the natural map
\[
\int_{\cF,\eta}:\Omega(M,\cF)\longrightarrow \RR;\, \sigma\longmapsto \int_M \eta\wedge\sigma
\]
descends to a map $H^*(M,\cF)\to \RR$, see e.g.\ \cite[Proposition 3.5]{GoertschesNozawaToeben}.

\begin{lemma}\label{lem:nontrivialelementsbasiccohom}
For $k=0,\ldots,n$, the form $\omega^k$ defines a nonzero element in $H^{2k}(M,\cF)$.
\end{lemma}
\begin{proof}
It suffices to show the claim for $k=n$. The given form is $\cF$-basic and thus defines an element in $H^{2k}(M,\cF)$. To show that this element is nonzero we show that it maps to a nonzero real number under the above integration operator, with respect to an appropriate relatively closed $s$-form.

The form $\eta_{1}\wedge \cdots \wedge \eta_s$ is relatively closed with respect to the foliation $\cF$: we have
\[
d(\eta_{1}\wedge\cdots\wedge \eta_s) = \sum_{i=1}^s (-1)^{i-1}\omega \wedge \eta_1\wedge \cdots \wedge \widehat{\eta_i}\wedge \cdots \wedge \eta_s,
\]
and $\omega$ vanishes on vector fields tangent to $\cF$. Note also that $M$ admits a natural orientation induced by the volume form $\omega^n\wedge \eta_1\wedge\cdots \wedge\eta_s$. Then
\begin{align*}
\int_{\cF,\eta_{1}\wedge\cdots \wedge \eta_s} \omega^n = \int_M \omega^n\wedge \eta_1\wedge \cdots \wedge \eta_s\neq 0.
\end{align*}
\end{proof}
For $s=1$, i.e., in the $K$-contact setting, the following theorem was known previously -- the statement about the minimal number of closed leaves generalizes \cite[Corollary 1]{rukimbira}, and the equivalence of the four conditions results from \cite{GoeNozToe}. 
\begin{theorem}\label{thm:closedleaves}
The characteristic foliation of a compact metric $f$-$K$-contact manifold $M^{2n+s}$ has at least $n+1$ closed leaves.  If it has only finitely many closed leaves, then the following conditions are equivalent:
\begin{itemize}
\item The number of closed leaves of $\cF$ is $n+1$.
\item The basic cohomology $H^*(M,\cF)$ is that of $\CC P^n$, i.e., 
\[
H^*(M,\cF)= \RR[\omega]/([\omega^{n+1}]).
\]
\item The basic cohomology $H^*(M,\cF_{s-1})$ is that of a $2n+1$-dimensional sphere.
\item $M$ has the real cohomology ring of $S^{2n+1}\times T^{s-1}$.
\end{itemize}
\end{theorem}
\begin{proof}
If the number of closed leaves is finite, then it is, by Corollary \ref{cor:basiccohomclosedleaves}, given by $\dim H^*(M,\cF)$. But this vector space contains, by Lemma \ref{lem:nontrivialelementsbasiccohom}, the $n+1$ nontrivial elements $1,[\omega],\ldots,[\omega]^n$. This shows that there are at least $n+1$ closed leaves, and the equivalence of the first and second condition.

The equivalence of the second and third condition follows from the long exact Gysin-type sequence in Proposition \ref{prop:gysinsequences}. The equivalence of the third and fourth condition is Theorem \ref{thm:cohomproduct}.
\end{proof}
\begin{corollary}\label{cor:K}
Let $M^{2n+1}$ be a real cohomology sphere, equipped with a $K$-contact structure with finitely many closed Reeb orbits, and $\phi:M\to M$ an automorphism of the $K$-contact structure. Then $\phi$ sends every closed Reeb orbit to itself. 
\end{corollary}
\begin{proof}
As shown in Section \ref{sec:mappingtori} the mapping torus $M_\phi$ of $\phi$ naturally is a metric $f$-$K$-contact manifold. Let $C\subset M$ be the union of the closed Reeb orbits of $M$, which are exactly $n+1$ by Theorem \ref{thm:closedleaves}. Then the union of the closed leaves of the characteristic foliation $\cF$ of $M_\phi$ naturally is the mapping torus $C_\phi$, whose number of connected components is bounded from above by $n+1$, with equality if and only if $\phi$ sends every closed Reeb orbit to itself. But on the other hand every closed leaf of $\cF$ is isolated, hence this number of connected components is by Theorem \ref{thm:closedleaves} also bounded from below by $n+1$.
\end{proof}

\begin{example}
It was shown in \cite{GoeNozToe} that there is a $K$-contact structure with exactly four closed Reeb orbits on the $7$-dimensional Stiefel manifold $\SO(5)/\SO(3)$, which is a real cohomology sphere. All (iterated) mapping tori of automorphisms of this example thus satisfy the conditions in Theorem \ref{thm:closedleaves}.
\end{example}

\end{document}